\documentclass{amsart}
\usepackage{amssymb}
\usepackage{graphicx}

\newtheorem{thm}{Theorem}
\newtheorem{cor}{Corollary}
\newtheorem{lem}{Lemma}

\newtheorem{rem}{Remark}

\newcommand{\es}{{\mathcal S}}

\newcommand{\D}{{\mathbb D}}

\def\be{\begin{equation}}
\def\ee{\end{equation}}

\newcommand{\bee}{\begin{enumerate}}
\newcommand{\eee}{\end{enumerate}}

\newcommand{\blem}{\begin{lem}}
\newcommand{\elem}{\end{lem}}
\newcommand{\bthm}{\begin{thm}}
\newcommand{\ethm}{\end{thm}}
\newcommand{\bcor}{\begin{cor}}
\newcommand{\ecor}{\end{cor}}
\newcommand{\beg}{\begin{example}}
\newcommand{\eeg}{\end{example}}
\newcommand{\begs}{\begin{examples}}
\newcommand{\eegs}{\end{examples}}
\newcommand{\bdefe}{\begin{defin}}
\newcommand{\edefe}{\end{defin}}
\newcommand{\bprob}{\begin{prob}}
\newcommand{\eprob}{\end{prob}}
\newcommand{\bei}{\begin{itemize}}
\newcommand{\eei}{\end{itemize}}

\newcommand{\bcon}{\begin{conj}}
\newcommand{\econ}{\end{conj}}
\newcommand{\bcons}{\begin{conjs}}
\newcommand{\econs}{\end{conjs}}
\newcommand{\bprop}{\begin{propo}}
\newcommand{\eprop}{\end{propo}}
\newcommand{\br}{\begin{rem}}
\newcommand{\er}{\end{rem}}
\newcommand{\brs}{\begin{rems}}
\newcommand{\ers}{\end{rems}}
\newcommand{\bo}{\begin{obser}}
\newcommand{\eo}{\end{obser}}
\newcommand{\bos}{\begin{obsers}}
\newcommand{\eos}{\end{obsers}}
\newcommand{\bpf}{\begin{pf}}
\newcommand{\epf}{\end{pf}}
\newcommand{\ba}{\begin{array}}
\newcommand{\ea}{\end{array}}
\newcommand{\beq}{\begin{eqnarray}}
\newcommand{\beqq}{\begin{eqnarray*}}
\newcommand{\eeq}{\end{eqnarray}}
\newcommand{\eeqq}{\end{eqnarray*}}

\begin{document}
\bibliographystyle{amsplain}

\title[Two applications of Grunsky coefficients]{Two applications of Grunsky coefficients in the theory of univalent functions}

\author[M. Obradovi\'{c}]{Milutin Obradovi\'{c}}
\address{Department of Mathematics,
Faculty of Civil Engineering, University of Belgrade,
Bulevar Kralja Aleksandra 73, 11000, Belgrade, Serbia}
\email{obrad@grf.bg.ac.rs}

\author[N. Tuneski]{Nikola Tuneski}
\address{Department of Mathematics and Informatics, Faculty of Mechanical Engineering, Ss. Cyril and Methodius
University in Skopje, Karpo\v{s} II b.b., 1000 Skopje, Republic of North Macedonia.}
\email{nikola.tuneski@mf.edu.mk}





\begin{abstract}
Let $\mathcal{S}$ denote the class of functions $f$ which are analytic and univalent in the unit disk $\D=\{z:|z|<1\}$ and normalized with  $f(z)=z+\sum_{n=2}^{\infty} a_n z^n$.
Using a method based on Grusky coefficients we study  two problems over the class $\mathcal{S}$: estimate of the fourth logarithmic coefficient and upper bound of the coefficient difference $|a_5|-|a_4|$.
\end{abstract}

\maketitle

\medskip

\section{Introduction and definitions}

\medskip

Let $\mathcal{A}$ be the class of functions $f$ which are analytic  in the open unit disc $\D=\{z:|z|<1\}$ of the form
\be\label{e1}
f(z)=z+a_2z^2+a_3z^3+\cdots,
\ee
and let $\mathcal{S}$ be the subclass of $\mathcal{A}$ consisting of functions that are univalent in $\D$.

\medskip
For $f\in \mathcal{S}$ the logarithmic coefficients, $\gamma_n$, are defined by
\be\label{e2} \log\frac{f(z)}{z}=2\sum_{n=1}^\infty \gamma_n z^n.
\ee
Relatively little exact information is known about these coefficients. The natural conjecture $|\gamma_n|\le1/n$, inspired by the Koebe function (whose logarithmic coefficients are $1/n$) is false even in order of magnitude (see Duren \cite[Section 8.1]{duren}).
For the class $\mathcal{S}$ the sharp estimates of single logarithmic coefficients  are known only for $\gamma_1$ and $\gamma_2$, namely,
\[|\gamma_1|\le1\quad\mbox{and}\quad |\gamma_2|\le \frac12+\frac1e=0.635\ldots,\]
and are unknown for $n\ge3$. In \cite{MONT-thai} the authors  gave the estimate $|\gamma_3|\le0.5566178\ldots$ for the class $\mathcal{S}$. In this paper for the same class we give the estimation $|\gamma_4|\le 0.51059\ldots$.
For the subclasses of univalent functions  the situation is not a great deal better. Only the estimates of the initial logarithmic coefficients are available. For details see \cite{cho}.

\medskip
Another problem is finding sharp upper and lower bounds of the coefficient difference $|a_{n+1}|-|a_n|$ over the class of univalent functions. Since  the Keobe function has coefficients $a_n=n$, it is natural to conjecture  that $||a_{n+1}|-|a_n||\le1$. But this is false even when $n=2,$ due to Fekete and Szeg\"o (\cite{FekSze33}) who obtained the sharp bounds
 $$ -1 \leq |a_3| - |a_2| \leq \frac{3}{4} + e^{-\lambda_0}(2e^{-\lambda_0}-1) = 1.029\ldots, $$
where $\lambda_0$ is the unique solution of the equation $4\lambda = e^{\lambda}$ on the interval $(0,1)$.
Hayman in \cite{Hay63} showed that if $f \in {\mathcal S}$, then $| |a_{n+1}| - |a_n| | \leq C$, where $C$ is an absolute constant
and the best estimate of $C$ is  $3.61\ldots$ (\cite{Gri76}). In the case when $n=3$ in \cite{MONT-thai}, the authors improved this to $1.751853\ldots$.
In this paper we also consider the difference $ |a_{5}| - |a_{4}| $.
\medskip

For the study of the problems defined above  we will use method based on Grunsky coefficients.
In the proofs we will use mainly the notations and results given in the book of N. A. Lebedev (\cite{Lebedev}).

\medskip

Here are basic definitions and results.

\medskip

Let $f \in \mathcal{S}$ and let
\[
\log\frac{f(t)-f(z)}{t-z}=\sum_{p,q=0}^{\infty}\omega_{p,q}t^{p}z^{q},
\]
where $\omega_{p,q}$ are so called Grunsky's coefficients with property $\omega_{p,q}=\omega_{q,p}$.
For those coefficients we have the next Grunsky's inequality (\cite{duren,Lebedev}):
\be\label{eq 3}
\sum_{q=1}^{\infty}q \left|\sum_{p=1}^{\infty}\omega_{p,q}x_{p}\right|^{2}\leq \sum_{p=1}^{\infty}\frac{|x_{p}|^{2}}{p},
\ee
where $x_{p}$ are arbitrary complex numbers such that last series converges.

\medskip

Further, it is well-known that if $f$ given by \eqref{e1}
belongs to $\mathcal{S}$, then also
\be\label{eq4}
f_{2}(z)=\sqrt{f(z^{2})}=z +c_{3}z^3+c_{5}z^{5}+\cdots
\ee
belongs to the class $\mathcal{S}$. So, for the function $f_{2}$ we have the appropriate Grunsky's
coefficients of the form $\omega_{2p-1,2q-1}^{(2)}$ and inequality \eqref{eq 3} reaches the form:
\be\label{eq5}
\sum_{q=1}^{\infty}(2q-1) \left|\sum_{p=1}^{\infty}\omega_{2p-1,2q-1}x_{2p-1}\right|^{2}\leq \sum_{p=1}^{\infty}\frac{|x_{2p-1}|^{2}}{2p-1}.
\ee

\medskip

Here, and further in the paper we omit the upper index "(2)" in  $\omega_{2p-1,2q-1}^{(2)}$ if compared with Lebedev's notation.

\medskip

From inequality \eqref{eq5}, when $x_{2p-1}=0$ and $p=3,4,\ldots$, we have
\begin{equation}\label{e6}
\begin{split}
&\quad |\omega_{11} x_1 +\omega_{31} x_3 |^2 +3|\omega_{13} x_1 +\omega_{33} x_3 |^2 + 5|\omega_{15} x_1 +\omega_{35} x_3 |^2 \\
&+7|\omega_{17} x_1 +\omega_{37} x_3 |^2\leq |x_1|^2+\frac{|x_3|^2}{3}.
\end{split}
\end{equation}

\medskip

As it has been shown in \cite[p.57]{Lebedev}, if $f$ is given by \eqref{e1} then the coefficients $a_{2}$, $ a_{3}$, $ a_{4}$ and $a_5$ are expressed by Grunsky's coefficients  $\omega_{2p-1,2q-1}$ of the function $f_{2}$ given by
\eqref{eq4} in the following way:
\be\label{e7}
\begin{split}
a_{2}&=2\omega _{11},\\
a_{3}&=2\omega_{13}+3\omega_{11}^{2}, \\
a_{4}&=2\omega_{33}+8\omega_{11}\omega_{13}+\frac{10}{3}\omega_{11}^{3},\\
a_{5}&=2\omega_{35}+8\omega_{11}\omega_{33}+5\omega_{13}^{2}+18\omega_{11}^2\omega_{13}+\frac73\omega_{11}^4,\\
0&= 3\omega_{15}-3\omega_{11}\omega_{13}+\omega_{11}^3-3\omega_{33},\\
0&=\omega_{17}-\omega_{35}-\omega_{11}\omega_{33}-\omega_{13}^{2}+\frac{1}{3}\omega_{11}^{4}.
\end{split}
\ee

\medskip
We note that in the cited book of Lebedev there is a typing mistake for the coefficient $a_{5}$. Namely, instead of the term $5\omega_{13}^{2}$  there stays $5\omega_{15}^{2}$.

\medskip

We now give upper bound of the fourth logarithmic coefficient over the class $\es$.

\medskip

\begin{thm}\label{th1}
Let $f\in\mathcal{S}$ and be given by \eqref{e1}. Then
\[|\gamma_4| \le 0.51059\ldots.\]
\end{thm}

\medskip

\begin{proof}
From \eqref{e1} and \eqref{e2}, after differentiating and comparing  coefficients, we receive
\[\gamma_4 =\frac{1}{2}\left(a_5-a_2a_4-\frac{1}{2}a_3^2+a_{2}^{2}a_{3}-\frac{1}{4}a_{2}^{4}\right),\]
or by using the relation \eqref{e7}:
\be\label{e8}
\gamma_4 =\frac{1}{2}\left(2\omega_{35}+3\omega_{13}^{2}+4\omega_{11}\omega_{33}+4\omega_{11}^{2}\omega_{13}
-\frac{5}{6}\omega_{11}^{4}\right).
\ee
If we combine the two last relations from \eqref{e7}, then we have
\be\label{e9}
\omega_{33} =  \omega_{15}-\omega_{11}\omega_{13}+\frac{1}{3}\omega_{11}^3
\ee
and
\be\label{e10}
\omega_{35}=\omega_{17}-\omega_{11}\omega_{33}-\omega_{13}^{2}+\frac{1}{3}\omega_{11}^{4}
=\omega_{17}-\omega_{11}\omega_{15}+\omega_{11}^{2}\omega_{13}-\omega_{13}^{2}.
\ee
Using the relations \eqref{e8}, \eqref{e9}and \eqref{e10}, after some calculations, we get
\[\gamma_4=\omega_{17}+\omega_{11}\omega_{15}+\omega_{11}^{2}\omega_{13}+
\frac{1}{2}\omega_{13}^{2}+\frac{1}{4}\omega_{11}^{4}.\]
Therefore,
\be\label{e11}
\begin{split}
|\gamma_4|&\leq |\omega_{17}|+|\omega_{11}||\omega_{15}|+|\omega_{11}|^{2}|\omega_{13}|+
\frac{1}{2}|\omega_{13}|^{2}+\frac{1}{4}|\omega_{11}|^{4}\\
&:=\varphi(|\omega_{11}|,|\omega_{13}|,|\omega_{15}|,|\omega_{17}|) .
\end{split}
\ee
Now, choosing $x_1=1$ and $x_3=0$ in \eqref{e6} we receive
\[ |\omega_{11} |^2 +3|\omega_{13} |^2 + 5|\omega_{15}|^2+7|\omega_{17}|^2 \leq 1, \]
and also
\[|\omega_{11}|\leq 1, \quad |\omega_{11}|^2 +3|\omega_{13} |^2\leq 1,\quad
|\omega_{11} |^2 +3|\omega_{13} |^2 + 5|\omega_{15}|^2\leq 1 . \]
The above inequalities imply
\be\label{e12}
\begin{split}
|\omega_{13}|&\le\frac{1}{\sqrt{3}}\sqrt{1-|\omega_{11}|^2}\,,\\
 |\omega_{15}|&\le\frac{1}{\sqrt{5}}\sqrt{1-|\omega_{11}|^2-3|\omega_{13}|^2}\,,\\
 |\omega_{17}| &\le\frac{1}{\sqrt{7}}\sqrt{1-|\omega_{11}|^2-3|\omega_{13}|^2-5|\omega_{15}|^{2}}\,.
\end{split}
\ee
Using \eqref{e11} and \eqref{e12} we conclude that it remains  to find $\max \varphi_{1}$, where
\[
\varphi_1(x,y,z,t)=\frac{1}{4}x^{4}+\frac{1}{2}y^{2}+x^{2}y+xz+t,
\]
where $(x,y,z,t)$ is in the four dimensional hypercube $\Omega$ described with
\begin{equation}\label{xyzt}
\begin{split}
0&\leq x=|\omega_{11}|\leq1,\\
0&\leq y=|\omega_{13}|\leq\frac{1}{\sqrt{3}}\sqrt{1-x^{2}},\\
0&\leq z=|\omega_{15}|\leq\frac{1}{\sqrt{5}}\sqrt{1-x^{2}-3y^{2}},\\
0&\leq t=|\omega_{17}|\leq\frac{1}{\sqrt{7}}\sqrt{1-x^{2}-3y^{2}-5z^{2}}.
\end{split}
\end{equation}

\medskip

Since $\varphi_1$ is an increasing function of $t$ on the interval $(0,+\infty)$, we realize that it reaches its maximal value for $t=t_0=\frac{1}{\sqrt{7}}\sqrt{1-x^{2}-3y^{2}-5z^{2}}$, i.e.,
\[ \max\{\varphi_1(x,y,z,t): (x,y,z,t)\in\Omega\} = \max\{\psi_1(x,y,z): (x,y,z)\in\Omega_1\}, \]
where $\psi_1(x,y,z)\equiv \varphi_1(x,y,z,t_0)$ and
\[ \Omega_1 = \left\{(x,y,z): 0\le x\le1,\, 0\le y\le \frac{1}{\sqrt3}\sqrt{1-x^2},\, 0\le z\le \frac{1}{\sqrt5}\sqrt{1-x^2-3y^2}\right\}. \]
The system of equations
\[
\left\{
\begin{split}
\frac{\partial \psi_1}{\partial x} &= x^3+2 x y+z-\frac{x/\sqrt{7}}{ \sqrt{1-x^2-3 y^2-5 z^2}}=0\\
\frac{\partial \psi_1}{\partial y} &= x^2+y-\frac{3 y/\sqrt{7}}{ \sqrt{1-x^2-3 y^2-5 z^2}}=0\\
\frac{\partial \psi_1}{\partial z} &= x-\frac{5 z/\sqrt{7}}{ \sqrt{1-x^2-3 y^2-5 z^2}}=0
\end{split}
\right.
\]
in the interior of $\Omega_1$, has a unique solution,
\[x_0 = 0.81907\ldots,\quad  y_0 = 0.233235\ldots, \quad z_0 = 0.126778\ldots, \]
(obtained with Wolfram's Mathematica) with $\psi_1(x_0,y_0,z_0,t_0) = \psi_1(x_0,y_0,z_0) = 0.51059\ldots$ which will turn out to be the maximal value of $\varphi_1$ on $\Omega$ and an upper bound for $|\gamma_4|$.

\medskip

Now we will study the behaviour of $\psi_1$ on the boundaries of $\Omega_1$.

\medskip

For $x=0$, we have that $\psi_1(0,y,z)=\frac12y^2+\frac{1}{\sqrt7}\sqrt{1-3y^2-5z^2}$ for $0\le y\le \frac{1}{\sqrt3}$ and $0\le z\le \frac{1}{\sqrt5}\sqrt{1-3y^2}$ has maximal value $\frac{1}{\sqrt7}=0.37796\ldots$ attained when $y=z=0$.

\medskip

Next, for $x=1$, we have  necessarily  $y=z=0$, which leads to a maximal value $\psi_1(1,0,0)=1/4=0.25$.

\medskip

For the case $y=0$, we have $\psi_1(x,0,z)=\frac14x^4+xz+\frac{1}{\sqrt7}\sqrt{1-x^2-5z^2}$, with $0\le x\le 1$ and $0\le z\le \frac{1}{\sqrt5}\sqrt{1-x^2}$.
Further, for the solution $(x_1,z_1)$ of the system of equations
\[
\left\{
\begin{split}
\frac{\partial\psi_1(x_1,0,z_1)}{\partial x} &= x_1^3+z_1-\frac{x_1/\sqrt{7}}{ \sqrt{-x_1^2-5 z_1^2+1}}=0\\
\frac{\partial\psi_1(x_1,0,z_1)}{\partial z} &= x_1-\frac{5 z_1/\sqrt{7}}{\sqrt{-x_1^2-5 z_1^2+1}}=0
\end{split}
\right.,
\]
we have
\[ -5 z_1^2 -5 x_1^3 z_1+x_1^2 = 0,\]
leading further to
\[ z_1 = \frac{1}{10} \left(-5 x_1^3+\sqrt{5} \sqrt{5 x_1^6+4 x_1^2}\right).\]
Finally,
\[
\begin{split}
&\quad  \psi_1(x,0,z) \le \psi(x_1,0,z_1)\\
  &= -\frac{x_1^4}{4}+\frac{1}{2} \sqrt{x_1^6+\frac{4 x_1^2}{5}} x_1+\frac{1}{\sqrt{14}}\left[\sqrt{\left(-5 x_1^4+\sqrt{5} \sqrt{x_1^2 \left(5 x_1^4+4\right)} x_1-4\right) x_1^2+2} \right]
 \end{split}
\]
By the means of calculus of real functions of one real variables, one can verify that the last function
attains its maximum for $x=0.80210\ldots$ and $z=0.183847\ldots$, and that maximum is $0.414666\ldots$.

\medskip

The case $y=\frac{1}{\sqrt{3}}\sqrt{1-x^{2}}$, leads to $z=0$, and further to the function
\[ \psi_1\left(x,\frac{1}{\sqrt{3}}\sqrt{1-x^{2}},0\right) = \frac{x^4}{4}+\frac{x^2\sqrt{1-x^2} }{\sqrt{3}}+\frac{1}{6} \left(1-x^2\right)\]
with maximum $0.4000\ldots$ for $x=0.8874\ldots$ and $y=0.2661\ldots$.

\medskip

For $z=0$, we have $\psi_1(x,y,0) = \frac{x^4}{4}+x^2 y+\frac{y^2}{2}+\frac{1}{\sqrt{7}}\sqrt{-x^2-3 y^2+1}$ and working in the similar way as in the case $y=0$, we receive its maximum on $0\le x\le1$ and $0\le y\le \frac{1}{\sqrt3}\sqrt{1-x^2}$ to be $0.4561\ldots$ for $x=0.8358\ldots$ and $y=0.2619\ldots$.

\medskip

Finally, in a similar way as before, for the case $z=\frac{1}{\sqrt{5}}\sqrt{1-x^{2}-3y^{2}}$, by means of calculus, we can verify that the maximal value is $0.4570\ldots$ obtained for $x= 0.864969\ldots$ and $y = 0.239789\ldots$.
\end{proof}

\medskip

We now give upper bound of $|a_5|-|a_4|$ over the class $\es$.

\medskip

\begin{thm}
Let $f\in\es$ and be given by \eqref{e1}. Then
\[|a_5|-|a_4| \le 2.3297\ldots.\]
\end{thm}

\begin{proof}
Since
\[
\begin{split}
|a_5|-|a_4| &\leq |a_5|-|\omega_{11}||a_4| \leq  |a_5-\omega_{11} a_4| \\
&= \Big|2\omega_{35} +6\omega_{11}\omega_{33} + 10\omega_{11}^{2}\omega_{13}+5\omega_{13}^{2}- \omega_{11}^4\Big|,
\end{split}
\]
after applying \eqref{e9} and \eqref{e10}, and some calculations, we have
\[
\begin{split}
|a_5|-|a_4|
&\leq\Big|2\omega_{17} +4\omega_{11}\omega_{15} + 6\omega_{11}^{2}\omega_{13}+3\omega_{13}^{2} +\omega_{11}^4\Big|\\
&\leq 2|\omega_{17}| +4|\omega_{11}||\omega_{15}| + 6|\omega_{11}|^{2}|\omega_{13}|+3|\omega_{13}|^{2} +|\omega_{11}|^4\\
&:=\varphi_2(|\omega_{11}|,|\omega_{13}|,|\omega_{15}| ,|\omega_{17}|),
\end{split}
\]
where
\[ \varphi_2(x,y,z,t) = x^{4}+3y^{2}+6x^{2}y+4xz+2t, \]
with $x,y,z,t$, as well as their domain $\Omega$,  are given in \eqref{xyzt} from the previous theorem.

\medskip

Now, in a similar way as in the proof of the previous theorem we will find the maximal value of the function $\psi_2$ over the domain $\Omega$.

\medskip

The function $\varphi_2$ is an increasing one over the variable $t$, and therefore it reaches its maximal value for $t=t_0=\frac{1}{\sqrt{7}}\sqrt{1-x^{2}-3y^{2}-5z^{2}}$.
Using the notations
\[\psi_2(x,y,z)\equiv\varphi_2(x,y,z,t_0) = x^4+6 x^2 y+4 x z+3 y^2+\frac{2 \sqrt{-x^2-3 y^2-5 z^2+1}}{\sqrt{7}},\]
and $\Omega_1$ as in the Theorem \ref{th1},
again using Wolfram's Mathematica we obtain that the system of equations
\[
\left\{
\begin{split}
\frac{\partial \psi_2}{\partial x} &= 4 x^3+12 x y+4 z-\frac{2 x}{\sqrt{7} \sqrt{1-x^2-3 y^2-5 z^2}} =0\\
\frac{\partial \psi_2}{\partial y} &= 6 x^2 + y \left(6-\frac{6}{\sqrt{7} \sqrt{1-x^2-3 y^2-5 z^2}}\right)=0\\
\frac{\partial \psi_2}{\partial z} &= 4 x-\frac{10 z}{\sqrt{7} \sqrt{1-x^2-3 y^2-5 z^2}}=0
\end{split}
\right.,
\]
in the interior of $\Omega_1$ has a unique solution,
\[x_0 = 0.82745\ldots,\quad  y_0 = 0.29092\ldots, \quad z_0 = 0.098698\ldots\ldots, \]
such that $\varphi_2(x_0,y_0,z_0,t_0) = \psi_2(x_0,y_0,z_0) = 2.3297\ldots$. At the end, this will turn out to be the maximal value of $\psi_2$ on $\Omega_1$ and upper bound of $|a_5|-|a_4|$.

\medskip

Now we will study the behaviour of $\psi_2$ on the boundaries of $\Omega_1$.

\medskip

For $x=0$, we receive $\psi_2(0,y,z)=3 y^2+\frac{2 \sqrt{-3 y^2-5 z^2+1}}{\sqrt{7}}$ which is a decreasing function of $z$ (since $z$ is  positive), thus with the same maximal value as $\psi_2(0,y,0) = 3 y^2+\frac{2 \sqrt{1-3 y^2}}{\sqrt{7}}$ which turns out to be $1.142857\ldots$ for $y=0.5345\ldots$.

\medskip

If $x=1$, then necessarily $y=z=0$, and $\psi_2(1,0,0)=1$.

\medskip

For $y=0$, we have $\psi_2(x,0,z)= x^4+4 x z+\frac{2 \sqrt{1-x^2-5 z^2}}{\sqrt{7}}$ which can be shown to have no critical points in the interior of
\[ \left\{ (x,z) : 0\le x\le1, 0\le z\le  \frac{1}{\sqrt{5}}\sqrt{1-x^2}\right\}, \]
and a maximal value $1.3614\ldots$ obtained for $x_1=0.9181\ldots$ and $z_1= \frac{1}{\sqrt{5}}\sqrt{1-x_1^2} = 0.1772\ldots$.

\medskip

For the case $y^*=\frac{1}{\sqrt{3}} \sqrt{1-x^2}$, we have $z=0$, and further, the function
\[ \psi_2\left(x, y^*, 0 \right) = x^4+\left(2 \sqrt{3-3 x^2}-1\right) x^2+1\]
has maximal value $2.118588\ldots$ for $x=0.8427\ldots$.

\medskip

Next, for $z=0$ we receive
\[\psi_2(x,y,0) = x^4+6 x^2 y+3 y^2+\frac{2 \sqrt{-x^2-3 y^2+1}}{\sqrt{7}},\]
with a critical point $(x_2,y_2)$ ($x=0.83589\ldots$ and $y=0.3097\ldots$) in the interior of
\[ \left\{ (x,y) : 0\le x\le1, 0\le y\le  \frac{1}{\sqrt{3}}\sqrt{1-x^2}\right\}, \]
such that $\psi_2(x_2,y_2,0) = 2.162\ldots$. The boundaries of the above domain are already discussed above.

\medskip

Finally, if $z^* = \frac{1}{\sqrt{5}}\sqrt{-x^2-3 y^2+1}$, one can verify that the function
\[ \psi_2(x,y,z) = x^4+6 x^2 y+3 y^2+\frac{4 x \sqrt{-x^2-3 y^2+1}}{\sqrt{5}}, \]
has critical point $(x_3,y_3)$ with $x_3=0.8338\ldots$ and $y_3=0.2921\ldots$, such that $\psi_2(x_3,y_3,z^*) = 2.287\ldots$.

\medskip

All the above analysis brings us to the final conclusion that $\psi_2$ on $\Omega_1$ has a maximal value $2.3297\ldots$ obtained for $x=x_0$, $y=y_0$ and $z=z_0$.
\end{proof}

\end{document}